\documentclass[11pt]{amsart}
\usepackage{amssymb
}
\newcommand{\no}[1]{#1}
\renewcommand{\no}[1]{}  \newcommand{\upDelta}{\Delta} 

\no{\usepackage{times}\usepackage[
subscriptcorrection, slantedGreek, 
nofontinfo]{mtpro}
\renewcommand{\Delta}{\upDelta}
}

\date{\today}

\newtheorem{theorem}{Theorem}[section]
\newtheorem{proposition}{Proposition}[section]

\theoremstyle{remark}
\newtheorem{example}{Example}

\theoremstyle{definition}
\newtheorem{remark}{Remark}[section]

\DeclareMathOperator{\supp}{supp}
\DeclareMathOperator{\singsupp}{singsupp}

\DeclareMathOperator{\WF}{WF}
\DeclareMathOperator{\art}{\mathcal{M}}

\newcommand{\eps}{\varepsilon}

\newcommand{\R}{{\bf R}}
\newcommand{\Id}{\mbox{Id}}
\renewcommand{\r}[1]{(\ref{#1})}
\newcommand{\PDO}{$\Psi$DO}
\newcommand{\be}[1]{\begin{equation}\label{#1}}
\newcommand{\ee}{\end{equation}}

\renewcommand{\d}{\mathrm{d}}

\renewcommand{\i}{\mathrm{i}}

\newcommand{\bo}{\partial \Omega}

\title[SAR]{Is a curved flight path in SAR better than a straight one?}

\author[P. Stefanov]{Plamen Stefanov}
\address{Department of Mathematics, Purdue University, West Lafayette, IN 47907}
\thanks{First author partly supported by a NSF Grant DMS-0800428}

\author[G. Uhlmann]{Gunther Uhlmann}
\address{Department of Mathematics, University of Washington, Seattle, WA 98195, and UC Irvine, CA 92697}
\thanks{Second author partly supported by  NSF, a Senior Clay Award and Chancellor Professorship at UC Berkeley}

\usepackage{graphicx}
\graphicspath{%
    {converted_graphics/}
    {/}
}
\begin{document}
\begin{abstract}
In the plane, we study the transform $R_\gamma f$ of integrating a unknown function $f$ over circles centered at a given curve $\gamma$. This is a simplified model of SAR, when the radar is not directed but has other applications, like thermoacoustic tomography, for example. We study the problem of recovering the wave front set $\WF(f)$. If the visible singularities of $f$ hit $\gamma$ once, we show that the ``artifacts'' cannot be resolved. If $\gamma$ is a closed curve, we show that this is still true. On the other hand, if $f$ is known a priori to have singularities in a compact set, then we show that one can recover $\WF(f)$, and moreover, this can be done in a simple explicit way, using backpropagation for the wave equation. 
\end{abstract}

\maketitle

\section{Introduction} 
In Synthetic Aperture Radar (SAR) imaging a plane flies along a curve in $\R^3$ and collects data from the surface, that we consider flat in this paper.  A simplified model of this is to project the curve on the plane, call it $\gamma$; then the data are integrals of a unknown density function on the surface over circles with various radii centered at the curve. Then the model is the inversion of the circular transform 
\be{1.1}
R_\gamma f(r,p) = \int_{|x-p|=r} f(x)\, \d \ell(x),  \quad p\in \gamma, \; r\ge 0,
\ee
where $\d \ell(x)$ is the Euclidean arc-length measure, and the center $p$ is restricted to a given curve $\gamma(t)$. This transform has been studied extensively; injectivity sets for $R_\gamma$ on $C_0^\infty$ have been described in full \cite{AgranovskiQ96},  see also \cite{GailK05}. In particular, each non-flat curve, does not matter how small, is enough for uniqueness. In view of the direct relation to the wave equation, this transform, and its 3 dimensional analog, see section~\ref{sec_3D}, have been studies extensively as well and in particular in thermoacoustic tomography with constant acoustic  speed, see, e.g., \cite{AgrKuchKun2008, ArgKuchment07, Gaik_2010,  FinchHR_07, finchPR,   Finch_Rakesh_2006, FinchRakesh08,  Haltmeier05, KuchmentKun08, Patch04}. A related transform is studied in \cite{FeleaAKQ,GaikKQ}.

\begin{figure}[h] 
  \centering
  \includegraphics[bb=0 -1 561 369,width=2.9in,height=1.91in,keepaspectratio]{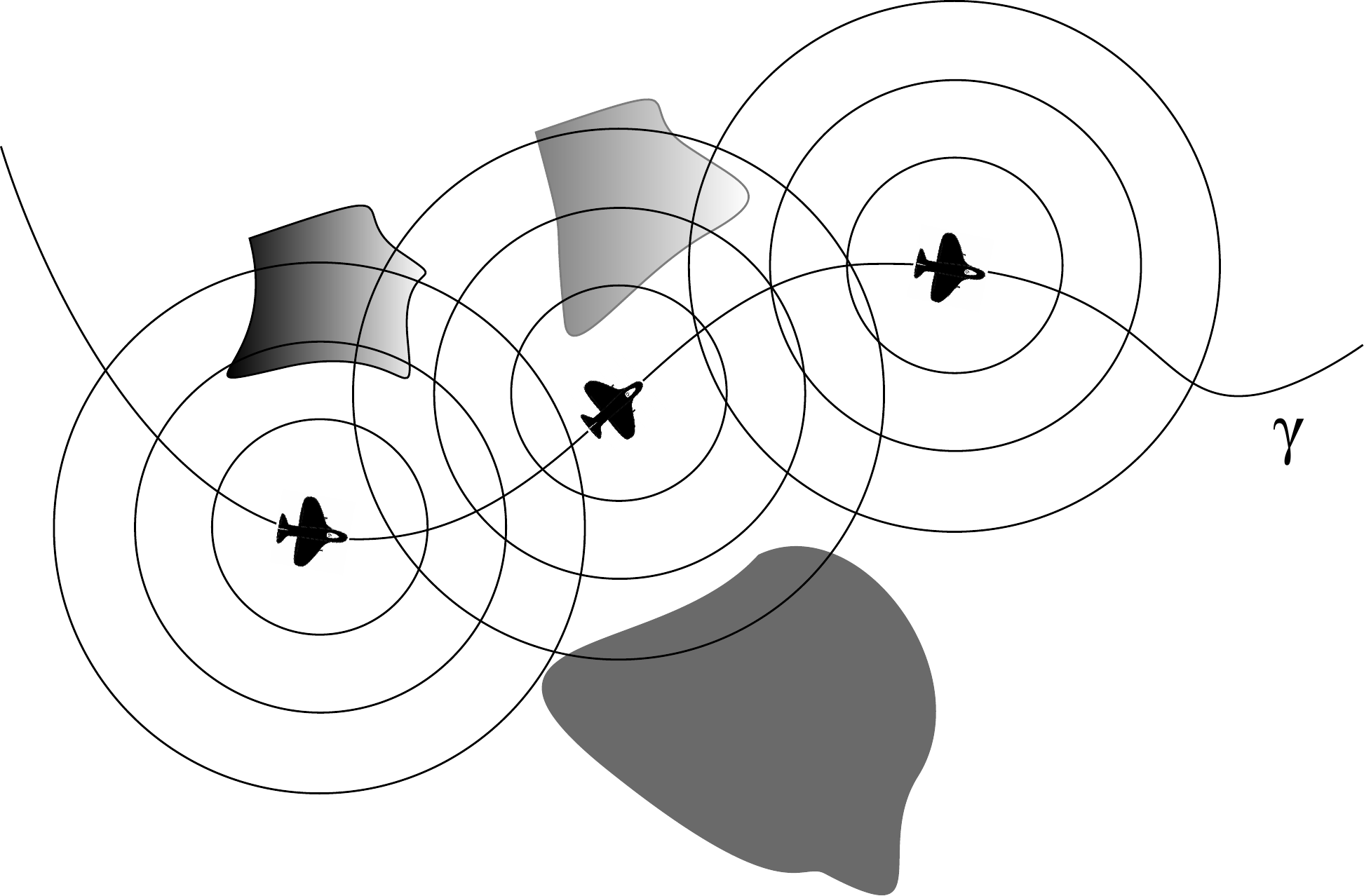}
 \caption{A plane surveying a flat surface}  \label{fig:sar1}
\end{figure}

The problem we study is the following: what part of the wave front set $\WF(f)$ can we recover? Clearly, we can only hope to recover the \textit{visible} singularities: those  conormal to  the circles involved in the transform, see also section~\ref{sec_kernel}. 

If $\gamma$ is a straight line, there is obvious non-uniqueness due to symmetry. Moreover, we can have cancellation of singularities symmetric about that line. More precisely, we can recover the singularities of the even part of $f$ and cannot recover those of the odd part. 

Based on this example, it has been suggested that a curved trajectory $\gamma$ might be a batter flight path. This question has been studied in \cite{NolanC2003}, and some numerical examples have been presented suggesting that when the curvature of $\gamma$ is non-zero, the artifacts are ``weaker'', and with increase of the curvature, they become even weaker. By artifacts, they mean  singularities in the wave front set of $R_\gamma^*R_\gamma f$ that are not in $\WF(f)$ located at \textit{mirror points}, see Figure~\ref{fig:sar2}. 
The same problem but formulated in terms of the wave equation model problem has been studied from a point of view of FIOs in \cite{ Nolan_Ch_04}, see also \cite{Felea07}, where the artifacts have been explained in terms of the Lagrangian of $R_\gamma$. They found that the artifacts are of the same strength, as an order of the corresponding FIO. More precisely, this is true at least away from the set of measure zero consisting of the points whose projections to the base falls on $\gamma$ (points right below the plane's path, i.e., $r=0$), and for $(x,\xi)$ such that the line trough it is tangent to $\gamma$ at some point. The latter set is responsible for existence of a submanifold of the Lagrangian near which the left and right projections are not diffeomorphisms. 
What part of the singularities of $f$ can be recovered however has not been studied, except for the cases when there is an amplitude which vanishes at the mirror points; then the artifacts can be ruled out by a prior knowledge. 

The main purpose of this paper is  two fold. First, we study the local problem --- what can be said about $\WF(f)$ knowing $\WF(R_\gamma)$ near some point, which localizes possible singularities of $f$ near two mirror points. More generally, we can assume that each line through $\WF(f)$ crosses $\gamma$ once, transversely. Then we show in Theorem~\ref{thm1} that curved trajectories $\gamma$ are no better than straight lines --- singularities can still cancel; moreover, the artifacts are unitary images of the original. We also describe microlocally the kernel of $R_\gamma $ modulo $C^\infty$. For simplicity, we stay away from the measure zero  set mentioned above. While this can be generalized globally for arbitrary curves, without the single intersection condition, we do not do this but study a closed curve encompassing a strictly convex domain. Then we show again that recovery of singularities is not possible. \textit{In this sense}, a curved or even a closed path is no better than a straight one. 

On the other hand, when $\gamma$ is closed and strictly convex, if we know a priori that $\WF(f)$ lies over a compact set (i.e., the projection of $\WF(f)$ onto the $x$-space is in a fixed compact set), then we show in Theorem~\ref{thm3} that one can recover the visible singularities. We even present a simple way to do that in the interior of $\gamma$, by backprojecting boundary data for the wave equation, see Proposition~\ref{pr_parametrix}.  In this sense, a curved trajectory is better. The effect which makes it possible is based on the fact that any singularity inside should be canceled by two outside if we see no singularities on the boundary; but the latter should be canceled by other singularities farther away, etc. At some point, this sequence would leave the compact set over which $\WF(f)$ lies a priori, thus contradicting the assumption on $f$. 

This transform belongs to the class of the X ray transforms with conjugate point studied by the authors in \cite{SU-caustics}. The circle centered at $\gamma$ and passing through $x$ in the direction $\theta$ has a conjugate point at the mirror image of $(x,\theta^\perp)$. The approach which we follow here is  different however.

\section{Main Results} 
Fix a smooth non self-intersecting curve  $(s_1,s_2)\ni s\mapsto \gamma(s)$.  For convenience, assume that $s$ is an arc-length parameter. 
We parameterize $R_\gamma f$ then by $s$ and the radius $r>0$, so we write $R_\gamma f(r,s)$ instead of $R_\gamma f(r,\gamma(s))$, compare with \r{1.1}.  
The \textit{possible} obstruction to recovery of singularities is well understood. Fix an orientation along $\gamma$ by choosing the normal field $\dot\gamma^\perp := (-\dot\gamma_2,\dot\gamma_1)$. This defines a ``Left" and a ``Right'' side of $\gamma$ near $\gamma$.  
Let $(x_L,\xi_L)\in T^*\R^2\setminus 0$, and $x_L\not\in \gamma$. Assume that the line through $(x_L,\xi_L)$ intersects $\gamma$ form the left, at some point $p_0=\gamma(s_0)$, and that this intersection is transversal.  If it is tangent, then the Lagrangian of $R_\gamma $ is not of a graph type, see e.g.\ \cite{Nolan_Ch_04}. We call such a singularity \textit{visible} from $\gamma$. We want to emphasize now that \textit{visible} does not necessarily mean  recoverable from $R_\gamma f$, which is the whole point of this paper.
Let $x_R$ be the point symmetric to $x_L$ about the line tangent to $\gamma$ at $p$ (a ``mirror'' point w.r.t. $p$), and let $\xi_R$ be the symmetric image of $\xi_L$, see Figure~\ref{fig:sar2}. Note that $\xi_{L}$, $\xi_R$  may both point towards $\gamma$, or both point  away from it. 
Then $(x_L, \xi_L)$ and $(x_R, \xi_R)$ are symmetric images to each other w.r.t.\ the symmetry about that tangent line lifted to the cotangent bundle. Denote this symmetry map by $\mathcal{C}$, i.e., $\mathcal{C} (x_L,\pm \xi_L) = (x_R,\pm \xi_R)$. 

Set $t_0 = |x_L-\gamma(s_0)|=|x_R-\gamma(s_0)|$.  The circular transform $R_\gamma f(r,s)$, for $(t,s)$ close to $(t_0,s_0)$ and  acting on a function $f$ supported in a small neighborhood of  $x_L$ and $y_L$ can only detect singularities close to $\pm\xi_L$ and $\pm \xi_R$  respectively, see section~\ref{sec_kernel}, but it is not clear if it can distinguish between them. We can expect that a singularity at $(x_L, \pm \xi_L)$   \textit{might} be cancelled by a singularity  at $(x_R, \pm\xi_R)$ and we \textit{might} not be able to resolve the visible singularities. 

\begin{figure}[h] 
  \centering
  \includegraphics[bb=0 -1 243 275,width=1.75in,height=1.99in,keepaspectratio]{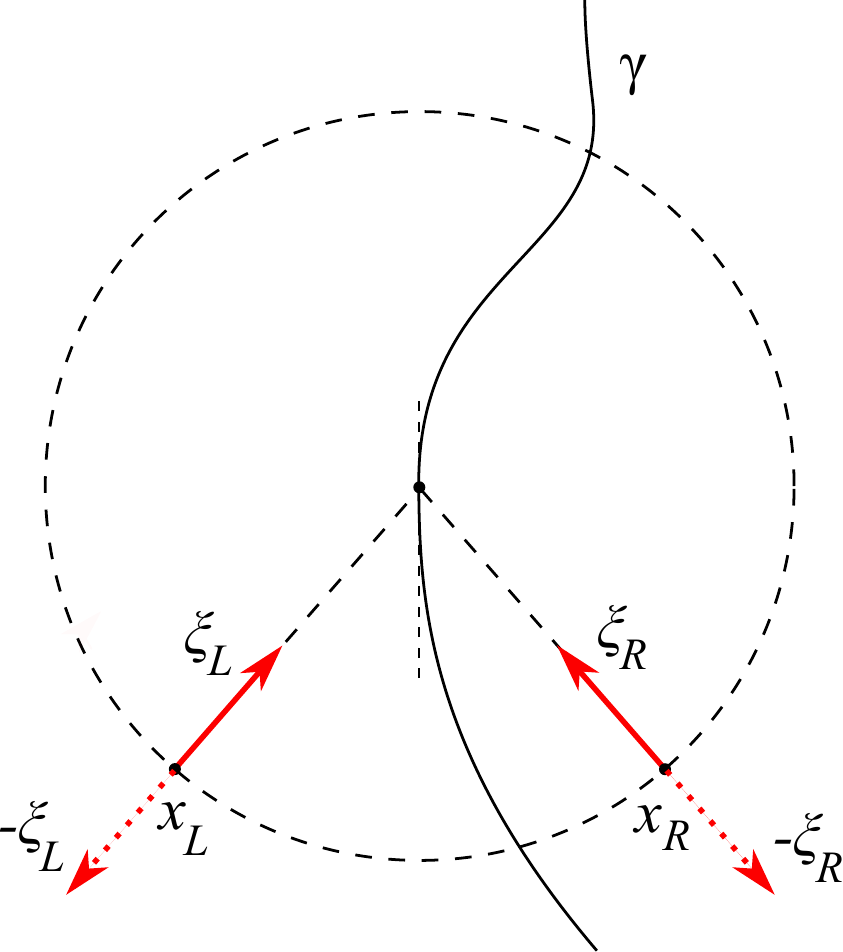}
  \caption{Mirror points: $(x_R, \pm\xi_R)$ are mirror points to $(x_L, \pm\xi_L)$, and vice versa  }
  \label{fig:sar2}
\end{figure}

Any open conic set in $T^*\R^2\setminus 0$ satisfying the  assumptions so far (also implied by the assumption below), can be written naturally as the union $\Sigma_L\cup \Sigma_R$ of two sets satisfying 
\be{A1}
\begin{split}
&\text{For any $(x,\xi)\in  \Sigma_{L}$ (or $\Sigma_R$), the line through $(x,\xi)$ hits $\gamma$ }\\ &\text{transversely from the left (right), at exactly one point different from $x$.}
\end{split}
\ee
Condition \r{A1} implies that $\Sigma_L$, $\Sigma_R$ are unions of disjoint open sets: $\Sigma_L = \Sigma_L^+ \cup \Sigma_L^-$, $\Sigma_R = \Sigma_R^+ \cup \Sigma_R^-$ where the positive and the negative signs indicate that $x+t\xi$ hits $\gamma$ for $t>0$ and $t<0$, respectively. 
Then $\mathcal{C}: \Sigma_L^\pm\to \Sigma_R^\pm$. 
Let $f$ be a compactly supported distribution with $\WF(f)\subset \Sigma_{L}\cup\Sigma_R$. The question we study is: what can we say about $\WF(f)$, knowing $\WF(R_\gamma f)$? Since $R_\gamma $ is linear, it is enough to answer the following question: let $R_\gamma f\in C^\infty(\gamma\times\R_+)$ (or let be smooth microlocally only, in a certain conic set). What can we say about $\WF(f)$?  

Without loss of generality, we can assume that $\mathcal{C}(\Sigma_L)=\Sigma_R$. In section~\ref{sec_go} below, we show that $R_\gamma $, restricted to distributions with wave front sets in $\Sigma_{L,R}$ is an FIO associated with a canonical graph denoted by $\mathcal{C}_{L,R}$, respectively.  In particular, the projection $\pi(\mathcal{C}_L( x_L,\pm\xi_L))$ on the base is $(t_0,s_0)$, i.e., $t_0$ is the time it takes to get to $\gamma$ with unit speed, and $s_0$ corresponds to  the point $p_0$ where that line hits $\gamma$. Then we set 
\be{0.1}
\Sigma_\gamma := \mathcal{C}_L(\Sigma_L) = \mathcal{C}_R(\Sigma_R).
\ee 
The possible singularities of $R_\gamma f$ with $f$ as above can only be in $\Sigma_\gamma$.

\begin{theorem}\label{thm1} 
Assume \r{A1}. 
Let $\Sigma_{L}$, $\Sigma_R$, $\Sigma_\gamma$, $\mathcal{C}$ be as above, and let $f_L$, $f_R$ be compactly supported distributions with $\WF(f_L)\subset \Sigma_{L}$,  $\WF(f_R)\subset \Sigma_{R}$. Then there exists a unitary Fourier Integral Operator  $U$ with canonical relation   $\mathcal{C}$ so that 
\be{t1}
R_\gamma (f_L+f_R)\in C^\infty(\Sigma_\gamma) \quad \Longleftrightarrow \quad  f_R-U f_L\in C^\infty(\Sigma_R).
\ee
Moreover, $U = -\Lambda_R^{-1}\Lambda_L$, with $\Lambda_{L,R}$ described in section~\ref{sec_go} and Proposition~\ref{pr_inverse}. 
\end{theorem} 

The unitarity of $U$ above is considered in microlocal sense: $U^*U-\Id$ an $UU^*-\Id$ are smoothing in $\Sigma_L$ and $\Sigma_R$, respectively, where the adjoint is taken in $L^2$ sense. 

The practical implications of Theorem~\ref{thm1} is that under assumption \r{A1}, only the singularities of $f_R-Uf_L$ (or, equivalently, $U^*\!f_R-f_L$) can be recovered. We can think of it as the ``even part'' of $f$ in this case. In particular, for any $f_L\in \mathcal{D}'$ with $\WF(f_L)\subset \Sigma_L$ there exists $f_R\in \mathcal{D}'$ with    $\WF(f_R)\subset \Sigma_R$ so that $R_\gamma (f_L+f_R)\in C^\infty(\Sigma_\gamma)$. An explicit radial example illustrating this is presented in Example~\ref{Example1}. 
Thus the artifacts when using $R_\gamma^*R_\gamma f$ to recover $\WF(f)$ are not just a problem with that particular method; they are unavoidable, and they are a unitary image of the original, i.e., ``equal'' in strength. From that point of view, a curved path  is no better than a straight one. 

We study next the case where  the path $\gamma$ and $f$ are such that there are singularities $(x,\xi)$ of $f$ for which the line through them hits $\gamma$ \textit{more than once}. Of course, this can happen for a curved path only. Consider the examples in Figure~\ref{fig:sar4}, where each of the dashed lines intersects $\gamma$ at most twice. 
We assume that there are no more intersection points than shown. 
On the left, the trace that $(x,\xi)$ leave on $\gamma$ at $p_1$ can be canceled by its mirror image $(x_1,\xi_1)$ about (the tangent at) $p_1$. Equivalently, $(x_1,\xi_1)$ can create an artifact at $(x_1,\xi_1)$, and vise-versa; related by a unitary map. Similarly, the singularity on $\gamma$ caused by  $(x,\xi)$ at $p_{-1}$ can be canceled by its mirror image $(x_{-1},\xi_{-1})$ about $p_{-1}$. We assume there that the lines through  $(x_1,\xi_1)$ and $(x_{-1},\xi_{-1})$ do not intersect $\gamma$ again. 
 If we know that one of the three singularities cannot exists, then none does. In particular, we can recover $(x,\xi)$ if we know a priori that either $(x_1,\xi_1)$, or $(x_{-1},\xi_{-1})$ cannot be in $\WF(f)$. Without any prior knowledge, we cannot. On the right, all those five singularities can cancel if they are related by suitable unitary operators. If we know that one of them cannot be in $\WF(f)$, then none can. 

Notice that $p_{-1} p p_1p_2p_3$ is a geometric optics ray reflected by  $\gamma$. To obtain $(x_2,\xi_2)$, for example, we start from $(x,\xi)$ going along the broken path, and at any point between $p_2$ and $p_3$, we go back from the same distance but along a straight line. If we go along the broken ray past $p_3$ (not shown on the picture), and come back along a line the same distance, we end up at $x_3$. The point $x_{-1}$ can be obtained similarly, going in direction opposite to $\xi$. 

So far we assumed that each line appearing in the construction intersects $\gamma$ at most twice. If this is not true, the mirror points to $(x,\xi)$ form a directed graph. We will not study this case.

Assume now that $\gamma$ is a closed curve and it encompasses a strictly convex domain $\Omega$. The discussion above suggests the following. For any $(x,\xi)\in T^*\R^2\setminus 0$, let $\Phi^t(x,\xi) = (x+t\xi/|\xi|,\xi)$. Let $\Phi_\gamma^t$ be defined on $T^*\Omega\setminus 0$ in the same way for small $|t|$, then extended by reflection, etc. At the values of $t$ corresponding to reflections, we take the limit from the left. 
We call this path, extended for all positive and negative $t$, a broken line through $(x,\xi)$. Then all mirror points of $(x,\xi)$, where possible artifacts might lie,  are given by 
\be{Art}
\art(x,\xi) = \left\{\Phi^{-t}\circ \Phi^t_\gamma(x,\xi);\; t\in\R\right\}, \quad x\in \Omega.
\ee
This is a discrete set under our  assumption, see, e.g., \cite{Lazutkin}. In the examples in Figure~\ref{fig:sar4}, this set is finite in each case, consisting of $(x,\xi)$, $(x_{\pm1},\xi_{\pm1})$, etc. Since $\gamma$ is a closed curve now, in our case it is infinite. The next theorem says that if we have a priori knowledge that would allow us to rule out at least one of those artifacts, then we can recover a singularity at $(x,\xi)$. Otherwise --- we cannot.

\begin{figure}[h] 
  \centering
  \includegraphics[bb=0 55 779 212,width=6.4in,height=1.29in,keepaspectratio]{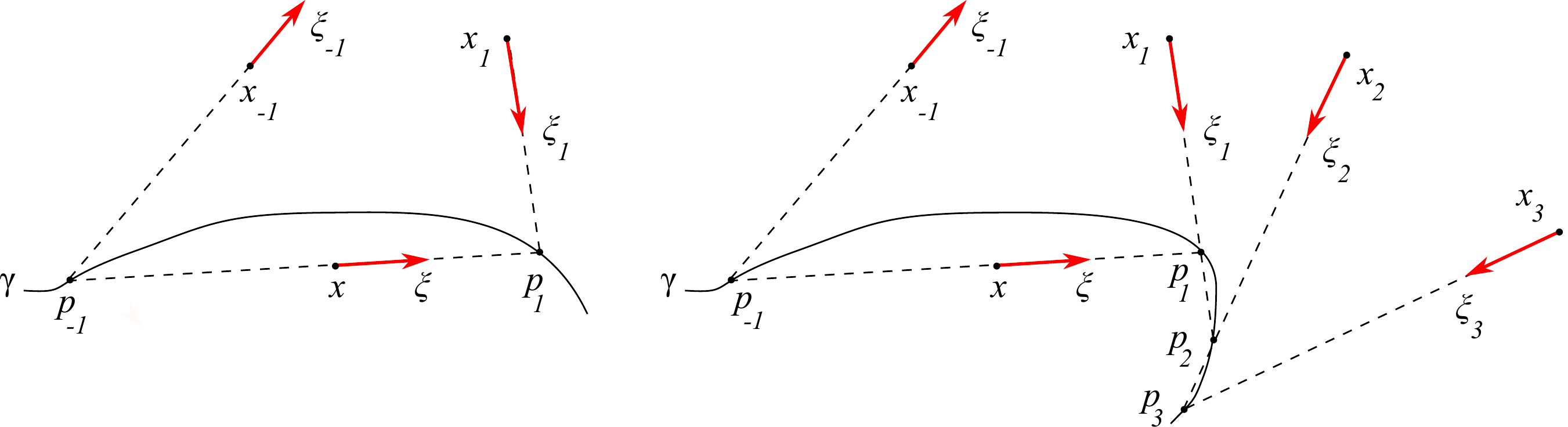}
\vspace{30pt}
  \caption{Singularities that cannot be resolved. Left: $(x,\xi)$ has mirror images  $(x_{-1},\xi_{-1})$ and $(x_1,\xi_1)$. Singularities at any two of those three points are related by unitary maps. Right: an example with more than three points. }
  \label{fig:sar4}
\end{figure}

Those arguments lead to the following ``propagation of singularities theorem''. 

\begin{theorem}\label{thm2} 
Let $\gamma=\bo$, where $\Omega\subset\R^2$ is a strictly convex domain. Let 
$f\in\mathcal{D}'(\R^2)$ with $\gamma\cap\supp f= \emptyset$, and assume that $R_\gamma f\in C^\infty$. Then for any $(x,\xi)\in T^*\Omega\setminus 0$, either $\art(x,\xi)\subset \WF(f)$  or $\art(x,\xi)\cap\WF(f)=\emptyset$.  
\end{theorem} 
As in the example above, if we know a priori that one of those points cannot be in $\WF(f)$, then none is, and in particular, $f$ is smooth at $(x,\xi)$. One such case is when $\WF(f)$  a priori lies over a fixed compact set. 

\begin{theorem}\label{thm3} 
Let $\gamma$ be as in Theorem~\ref{thm2}. 
Let $f\in\mathcal{E}'(\R^2)$. If $R_\gamma f\in C^\infty$, then $\WF(f)$ contains no singularities visible from $\gamma$, and in particular,  $f|_\Omega\in C^\infty$. Moreover, $f|_\Omega$ can be obtained from $R_\gamma f$ modulo $C^\infty$ by the back-projection operator described in Proposition~\ref{pr_parametrix}. 
\end{theorem} 

If we do not have a priori information about $f$, then $\WF(f)$ cannot be reconstructed. 

\begin{theorem}\label{thm4}
Let $\gamma$ be as in Theorem~\ref{thm2}. Then there is $f\in \mathcal{D}'(\R^2\setminus\gamma)\setminus C^\infty$  so that $R_\gamma f\in C^\infty(\R_+ \times\gamma)$. Moreover, for any $f$ with $\singsupp f\subset \Omega$, there is $g$ with $\singsupp g\subset\R^2\setminus\Omega$ so that $R_\gamma (f-g)\in C^\infty(\R_+\times\gamma)$.

\end{theorem} 

The second statement of the theorem says that   we can take any $f$ singular in $\Omega$,  and extend it outside $\Omega$ so that its circular transform will be smooth on $\gamma$. Therefore, not only singularities cannot be detected  but any chosen in advance $f$ singular in $\Omega$ can be neutralized by choosing suitable extension singular outside   $\Omega$. We refer also to Example~\ref{Example1} and the remark at the end of it for a radial example.

Those problems and  the methods are related to the thermoacoustic problem with sources inside and outside $\Omega$, see Remark~\ref{remark_TAT} .

\section{Proofs} \label{sec2}

\subsection{The wave front set of the kernel of $R_\gamma$} \label{sec_kernel} 
The Schwartz kernel of $R_\gamma $ is given by
\[
\mathcal{R}(r,s,x) =\frac1r\delta_{S^1}\left(\frac{x-\gamma(s)}{r}\right).
\]
The factor $1/r$ is not singular for $r>0$, where we work. By the calculus of wave front sets, 
\[
\WF(\mathcal{R}) = \left\{  (r,s,x,\d F^t_{r,s,x}\eta);\;   x-\gamma(s)\in r S^1, \; \eta = k\frac{x-\gamma(s)}{| x-\gamma(s)   |}, \; k\not=0 \right \},
\]
where $F=(x-\gamma(s))/r$. Set $\omega = F|_{S^1}$ to write this as
\[
\WF(\mathcal{R}) = \left\{ \left(r,s,\gamma(s)+r\omega,  -\frac{k}{r} , -\frac{k}{r}  \omega\cdot\dot\gamma(s) , \frac{k}r \omega \right)  ;\;  \omega\in  S^1, r>0, 
k\not=0  \right\}.
\]
Set $\xi=-k\omega/r$; then $\omega = -\eps \xi/|\xi|$, $k/r =\eps |\xi|$, 
where $\eps=\pm1$ is the sign of $k$, to get
\be{1.W}
\WF(\mathcal{R}) = \left\{ \left(r,s,x,  -\eps|\xi| , \xi\cdot\dot\gamma(s) , -\xi\right)  ;\; x+\eps r\frac{\xi}{|\xi|}=\gamma(s),  r>0, \eps=\pm1 \right\}.
\ee
By the calculus of wave front sets, if we invert the sign of the sixth component there, $\-\xi$, and consider $\WF(\mathcal{R})$ as a relation, this tells us where $\WF(f)$ is mapped under the action of $R_\gamma $. Comparing this with the definition  \r{go1b}, \r{go1c} of $\mathcal{C}_L$, and similarly for $\mathcal{C}_R$ below, we get 
\[
\WF(R_\gamma f)\subset \mathcal{C}_L(\WF(f)) \cup \mathcal{C}_R(\WF(f))
\]
for $f$ such that for any $(x,\xi)\in \WF(f)$, the line $x+s\xi$ through $(x,\xi)$ meets $\gamma$ exactly once, for $s\not=0$. 

Let $(\tau,\sigma)$ be the dual variables to $(r,s)$. The reason we use $\tau$ instead of the more intuitive choice $\rho$ for a dual variable to $r$ is that by applying the \PDO\  $A$ below, we will transform $r$ into a variable denoted by $t$. 
 Then 
\be{1c}
\WF(R_\gamma f)\subset \left\{ |\sigma|<|\tau|      \right\}.
\ee
Moreover, by \r{A1}, for any such $f$, we have $|\sigma|<\delta|\tau|$, $\delta<1$.

\subsection{Reduction to a problem for the wave equation}
Let $u$ solve the problem 
\begin{equation}   \label{1}
\left\{
\begin{array}{rcll}
(\partial_t^2 -\Delta)u &=&0 &  \mbox{in $\R_t\times \R_x^2$},\\
u|_{t=0} &=& 0,\\ \quad \partial_t u|_{t=0}& =&f, 
\end{array}
\right.               
\end{equation}
and set $\Lambda f=u|_{\R_+\times\gamma}$, i.e.,
\be{1.2a}
\Lambda f=  \frac{\sin(t|D|)}{|D|}f \Big|_{\R_+\times\gamma}.
\ee
The well known solution formula then implies
\be{1.3}
\Lambda f(t,s) = \int_0^t\frac{r R_\gamma f(r,s)}{\sqrt{t^2-r^2}}\, \d r,\quad  p\in \gamma.
\ee
Our assumptions imply that $R_\gamma f(r,p)=0$ for $0\le r\ll1$. The integral above then it has a kernel singular at the diagonal $t=r$ only. It belongs to the class of Abel operators 
\be{1.3a}
Ah (t) = \int_0^t\frac{r h(r)}{\sqrt{t^2-r^2}}\,\d r,\quad t>0.
\ee
Then 
\be{1.3bb}
\Lambda = (A\otimes \Id) R_\gamma,
\ee
in other words, $\Lambda$ is just  $AR_\gamma $ but $A$ acts in the first variable.  The explicit left inverse of $A$ is
\be{1.4}
h(r) = Bh(r) :=  \frac{2}{\pi r}\frac{\d}{\d r} \int_0^r\frac{ h(t)}{\sqrt{r^2-t^2}}\, \d t, \quad r>0.
\ee
\begin{proposition}\label{pr_AB}
The operator $A$ restricted to $\mathcal{E}'(\R_+)$ is an elliptic  \PDO\ of order $-1/2$ with principal symbol 
\[
\sigma_p(B)(r,\tau)=  \sqrt{\pi/2} e^{-\i \pi/4}\sqrt{r}\left(\tau_+^{-1/2}+ \i \tau_-^{-1/2}\right).
\]
The operator $B$ on $\mathcal{E}'(\R_+)$  is an elliptic  \PDO\ of order $1/2$ with principal symbol given by the inverse of that of $A$.
\end{proposition}
\begin{proof}
The Schwartz kernel of $A$ can be written as 
\[
A(t,r)= A^\sharp(t,r,t-r), \quad \text{where}\quad A^\sharp(t,r,w) = \frac{r}{\sqrt{t+r}}w_+^{-1/2},
\]
and $w_+=\max(w,0)$. The Fourier transform of $w_+^{-1/2}$ is equal to 
\[
\sqrt\pi e^{-\i \pi/4}\left(\tau_+^{-1/2}+ \i \tau_-^{-1/2}\right).
\]
Then $A$ is a formal \PDO\ with an amplitude given by the partial Fourier transform of $A^\sharp$ w.r.t.\ $w$, i.e., 
\[
 \sqrt\pi e^{-\i \pi/4}\frac{r}{\sqrt{t+r}}\left(\tau_+^{-1/2}+ \i \tau_-^{-1/2}\right).
\]
Since $t$ and $r$ are strictly positive, there is no singularity in $1/\sqrt{t+r}$. The singularity at $\xi=0$ can be cut off at the expense of a smoothing term. Set $t=r$ to get the principal symbol of $A$. Since $B$ is a parametrix of $A$, the second assertion follows directly. 
\end{proof} 

Note that the full symbol of $A$ can be computed from the asymptotic expansion of the Bessel function $J_0$ since $A$ is the composition of the Fourier Sine transform and the zeroth order Hankel transform, see \cite{Gorenflo_Abel}.

\subsection{Working with  the Darboux equation.}\label{Darboux}
The unrestricted spherical means $Gf(t,x) := (2\pi t)^{-1}Rf$ solve the Darboux equation
\[
\left(\partial_t^2 +\frac1t\partial_t-\Delta\right)Gf(t,x)=0
\]
with boundary conditions $Gf(0,x)=f(x)$, $\partial_tGf(0,x)=0$, see e.g., \cite{ArgKuchment07} and the references there. The Darboux equation has the same principal symbol as the wave equation and therefore the same propagation of singularities for $t\not=0$. Replacing the wave equation with the Darboux one seems as a natural thing to do --- this would have eliminated the need for the operators $A$ and $B$. 
On the other hand, $t=0$ is a singular point which poses technical problems with the backprojection, and for this reason we prefer to work with the wave equation.

\subsection{Geometric Optics}\label{sec_go}
 The solution of \r{1} is given by
\be{go1}
u =  \frac{\sin(t|D|)}{|D|}f = -\frac{e^{-\i t|D|}}{2\i |D|}f +  \frac{e^{\i t|D|}}{2\i |D|}f = u_+ +u_-,
\ee
where
\be{go1'}
\begin{split}
u_+&= \frac1{(2\pi)^{2}} \int e^{\i (( x-y)\cdot \xi -t|\xi|   )}\frac{(-1)}{2\i|\xi|} f(y)\, \d y\, \d\xi,\\ u_- &= \frac1{(2\pi)^{2}} \int e^{\i (( x-y)\cdot \xi +t|\xi|   )}\frac{1}{2\i|\xi|} f(y)\, \d y\, \d\xi.
\end{split}
\ee
The first term $u_-$ is in the kernel of $\partial_t+\i|D|$, and if we consider $t$ as a parameter, it is an FIO associated with the canonical relation $(x,\xi) \mapsto (x+t\xi/|\xi|,\xi)$. The second term $u_-$ is in the kernel of  $\partial_t-\i|D|$ associated with  $(x,\xi) \mapsto (x-t\xi/|\xi|,\xi)$. 

We assume now that $\WF(f)\subset \Sigma_L$, see \r{A1}. Then we set $\Lambda_L f=\Lambda f$ with $f$ as above. We define $\Lambda_R$ in a similar way. 

Restrict \r{go1} to $\R\times\gamma$, see \r{1.2a}, to get
\be{go1r}
\Lambda_L = \Lambda_L^+ +\Lambda_L^-, 
\ee
where $\Lambda_\pm f$ are the restrictions of the two terms above to $\R\times\gamma$. For  the first term,  we set $x=\gamma(s)$ to get
\be{go1a}
\Lambda_L^+ f :=  -\frac{e^{-\i t|D|}}{2\i |D|}f\Big|_{\R_+\times\gamma} = (2\pi)^{-2} \int e^{\i (( \gamma(s)-y)\cdot \xi -t|\xi|   )}\frac{(-1)}{2\i|\xi|} f(y)\, \d y\, \d\xi.
\ee
Since we made an assumption guaranteeing that $\dot\gamma(s)\cdot\xi\not=0$,  this is an elliptic FIO with a non-degenerate phase function, see e.g, \cite{Treves2}, Ch.VI.4 and Ch.VIII.6, of order $-1$ associated with the canonical relation
\[
\mathcal{C}_L^+ : (\gamma(s)-t\xi/|\xi|, \xi) \longmapsto (t,s,-|\xi|, \dot\gamma(s)\cdot\xi) ,
\]
well defined on $\Sigma_L^+$. 
Another way to write this is the following. Let $t(x,\xi)>0$, $s(x,\xi)$ be such that $x+t(x,\xi)\xi/|\xi|= \gamma(s(x,\xi))$. Then
\be{go1b}
\mathcal{C}_L^+ : (x, \xi) \longmapsto ( t(x,\xi), s(x,\xi),-|\xi|, \dot\gamma(s)\cdot\xi). 
\ee
Similarly, the second term in \r{go1} defines
\[
\Lambda_L^- f :=  \frac{e^{\i t|D|}}{2\i |D|}f\Big|_{\R_+\times\gamma} = (2\pi)^{-2} \int e^{\i (( \gamma(s)-y)\cdot \xi +t|\xi|   )}\frac{1}{2\i|\xi|} f(y)\, \d y\, \d\xi.
\]
This is an FIO associated with the canonical relation
\be{go1c}
\mathcal{C}_L^- : (x, \xi) \longmapsto ( t(x,-\xi) ,s(x,-\xi) ,|\xi|, \dot\gamma(s)\cdot\xi),
\ee
since   $x-t(x,-\xi)\xi/|\xi|=\gamma(s(x,-\xi))$ for  $(x,\xi)\in \Sigma_L^-$. 
We now define $\mathcal{C}_L$ as  $\mathcal{C}_L^+$ on $\Sigma_L^+$, and   $\mathcal{C}_L^-$ on $\Sigma_L^-$. Similarly, $\Lambda_{L} f$ is defined as  $\Lambda_{L}^+ f$ when $\WF(f)\in \Sigma_{L}^+$. 
Also, set $\Sigma = \mathcal{C}(\Sigma_L)\subset T^*(\R^+\times\gamma)$, see also \r{0.1}.

We define $\mathcal{C}_R^\pm$, $\mathcal{C}_R$, $\Lambda_R^\pm f$, $\Lambda_R f$ in the same way. In fact, they are the same maps as the ``L'' ones but restricted to $\Sigma_{R}^\pm$,  $\Sigma_R$, and $f$ with wave front sets there, respectively. Clearly, the map $\mathcal{C}$ defined in the Introduction satisfies
\[
\mathcal{C} = \mathcal{C}_R^{-1}\mathcal{C}_L : \Sigma_L \longrightarrow \Sigma_R, 
\]
and \r{0.1} holds.

Relations \r{go1b}, \r{go1c} imply also the following, compare with \r{1c},
\be{go3}
\WF(\Lambda_{L,R}^\pm f) \subset \left\{(t,s,\tau,\sigma); \;  |\sigma|\le \mp \delta\tau \right\},
\ee
where $0<\delta<1$ is the cosine of the smallest angle at which a line through $(x,\xi)\in \WF(f)$ can hit $\gamma$, see \r{go1b}, \r{go1c}.

Since $\Lambda_L$ and $\Lambda_R$ are elliptic FIOs (associated with canonical graphs), they have left and a right parametrices $\Lambda^{-1}_L$ and $\Lambda_R^{-1}$, of order $1$ associated with $\mathcal{C}_L^{-1}$ and $\mathcal{C}_R^{-1}$, respectively. We have the following more conventional representation of those inverses. 

We recall the definition of incoming and outgoing solutions in a domain $\Omega$. 
Let $u(t,x)$ solve the wave equation in $[0,T]\times\Omega$ up to smooth error, i.e., $(\partial_t^2-\Delta)u\in C^\infty$, where $\Omega\subset\R^2$ is a fixed domain, and $T>0$. We call $u$ \textit{outgoing} if $u(0,\cdot)=u_t(0,\cdot)=0$  in $\Omega$; and we call $u$ \textit{incoming} if $u(T,\cdot)=u_t(T,\cdot)=0$ in $\Omega$. We micro-localize those definitions as follows. A solution of the wave equation modulo smooth functions near $\R\times\gamma$, on the left (or right) of $\gamma$ is called outgoing/incoming, if all singularities starting from points on $\R\times\gamma$ propagate to the future only ($t>0$), and respectively to the past ($t<0$). 

\begin{proposition}\label{pr_inverse}
Let $u_L$ be the incoming solution of the wave equation with Dirichlet data $h$ on $\R_+\times\gamma$, where $\WF(h)\subset\Sigma$; and assume \r{A1}. Then 
\be{go2}
\Lambda_L^{-1}h = 2\partial_t u_L|_{t=0}. 
\ee
\end{proposition}

\begin{proof}
Call the operator on the r.h.s.\ of \r{go2}   $M$ for a moment. To compute $M\Lambda_L f$, recall \r{1.2a}.  Assume first that $\WF(f)\subset \Sigma_L^+$. Then $\Lambda_L f = \Lambda_L^+ f$, see \r{go1a}, i.e.,  $\Lambda_L f$ is the trace on the boundary of $u_+$ defined in \r{go1'}. Now, to obtain $M\Lambda_L f $, we have to find the incoming solution of the wave equation with boundary data $\Lambda_L f$. That solution would be $u_+$ modulo $C^\infty$, i.e., $u_+=u_L$ in this case. Then $M\Lambda_L f= 2\partial_tu_+|_{t=0}$, by the definition of $M$. The latter equals $f$, by the definition of $u_+$.  If $\WF(f)\subset \Sigma_L^-$, then  $\Lambda_L f = \Lambda_L^- f$, and $M\Lambda_L f= 2\partial_tu_-|_{t=0}=f$. In the general case, $f$ is a sum of two terms with wave front sets in $\Sigma_L^+$ and $\Sigma_L^-$, respectively. 

To see that $M$ is a right inverse as well (which in principle follows from the characterization of $\Lambda_L$ as an elliptic  FIO of graph type), let $u_L$ be as in the proposition. Then $Mh = 2 \partial_t u_L|_{t=0}$. To compute $\Lambda_L Mh$, we need to find  first the outgoing solution of the wave equation  with Cauchy data $(0,M h)$ at $t=0$. This solution must be $u_L$. Indeed, call that solution $v$ for the moment and write $v=v_++v_-$ as in \r{go1}. Assume first that $\WF(h)$ is included in $\tau<0$, where $\tau$ us the dual variable to $t$, see\r{go1b}. Then the singularities of $v_+$ hit $\R\times\gamma$ but those of $v_-$ do not, by \r{A1}. The solution $v_+$ has Cauchy data at $t=0$ given by 
\be{1.2'}
(-\frac1{2 \i} |D|^{-1} Mh, \frac12 Mh),
\ee
see \r{go1}. Now, $u_L$ has the same Cauchy data, which proves that $u_+=u_L$. Then $\Lambda_L Mh$ is the trace of $u_+$ on the boundary, which is $h$. The case $\tau>0$, and  the general one, can be handled in a similar way.
\end{proof}

\subsection{Proof of Theorem~\ref{thm1}} 
Set $f=f_L+f_R$. 
Assume now that $R_\gamma f\in C^\infty(\Sigma)$. Apply $A  \otimes \Id$ to that, where $A$ acts w.r.t. to $r$ and $\Id$ is w.r.t.\ $s$,   to get $\Lambda f\in C^\infty(\Sigma)$. Since $A$ has a left inverse on $\R_+$, this is actually equivalent to $R_\gamma f\in C^\infty(\Sigma)$, i.e., 
\be{1.3b}
R_\gamma f\in C^\infty (\Sigma)  \quad  \Longleftrightarrow\quad \Lambda f_L +\Lambda f_R\in C^\infty(\Sigma).
\ee
Indeed, recall that $(\tau,\sigma)$ is the dual variable to $(t,r)$; then  $A  \otimes \Id$ is elliptic on $\{\tau\not=0\}$. By \r{go3}, $A  \otimes \Id$ is elliptic in a conic neighborhood of $\WF(\Lambda f)$, which proves our claim. 
The restrictions of the wave front sets of $f_L$ and $f_R$ imply that we can replace $\Lambda$ above by its microlocalized versions $\Lambda_{L,R}$:
\be{1.5}
R_\gamma f\in C^\infty (\Sigma)  \quad  \Longleftrightarrow\quad   \Lambda_L f_L +\Lambda_R f_R\in C^\infty(\Sigma).
\ee

Now, apply the parametrix $\Lambda_L^{-1}$ to \r{1.5} to get
\be{1.6}
f_L   + \Lambda_L^{-1}\Lambda_R f_R\in C^\infty(\Sigma).
\ee
Of course, starting from \r{1.6} we can always go back to \r{1.5}. Therefore, \r{1.4} and \r{1.5} are equivalent, and they are both equivalent to $\Lambda_R^{-1}\Lambda_Lf_L   + f_R \in C^\infty(\Sigma)$.

To show that $U$ is unitary, we will compute $\|Uf_L\|$ first. Let $f_L$ be as above. Denote by $u_L$ the  solution with Cauchy data $(0,f_L)$ at $t=0$. To obtain $\Lambda_R^{-1}$, we need to solve backwards (to find the incoming solution) of the wave equation on the right of $\gamma$ with boundary data $\Lambda_+f = u_L|_{\R_+\times\gamma}$. Let us call that solution $u_R$. On the other hand, $u_L$ restricted to the right of $\gamma$ is an outgoing solution with the same trace on the boundary. Then $v:= u_R-u_L$ solves the wave equation  the right of $\gamma$ , and for $t=0$ we have $v=u_R$; while for $t=T\gg1$, we have $v=-u_L$. Moreover, $v$ has zero Dirichlet data on the boundary.  Therefore, up to a smoothing operator applied to $f_L$, the energy of $u_L$ at $t=T$ coincides with that of $u_R$ at $t=0$. The former one is equal to the energy of the Cauchy data $(0,f_L)$, up to smoothing operator, and therefore, $E(u_R(0)) = \|f_L + Kf_L\|_{L^2}$, where $K$ is smoothing. If $\WF(f_L)\subset \Sigma_L^+$, then $u_L$ solves $(\partial_t+\i|D|)u_L\in C^\infty$, and then so does $u_R$. Then $E(u_R(0)) = \|\Lambda_R^{-1}h\|^2_{L^2}$, see \r{1.2'}, where $h=\Lambda_Lf_L$. Therefore we showed that
\[
\|(\Id +K)f_L\|_{L^2}=  \|\Lambda_R^{-1}\Lambda_L f_L\|_{L^2}.
\]
This proves that $U^*U=\Id$ modulo an operator that is smoothing on $\Sigma_L^+$. In the same way we show that this holds on $\Sigma_L^-$ which is disconnected from $\Sigma_L^+$. Since $U$ is microlocally invertible on $\Sigma_L$, we get that $U$ is unitary up to a smoothing operator on  $\Sigma_L$, as claimed. 

This completes the proof of the theorem. 

\begin{example}\label{Example1} 
We give an example of cancellation of singularities. Let $\gamma= S^1$ be the unit circle  parameterized by its polar angle $s$. Then $|\dot\gamma|$ is not unit but constant, which is enough. Let $f$ be the characteristic function of the circle $|x|=1/2$, i.e., $f(x) = H(1/4-|x|^2)$, where $H$ is the Heaviside function. Then, clearly, $R_\gamma f(r,\theta)$ is singular at $r=1/2$ (not only), with a singularity of the type $\sqrt{r-1/2}$, see also Figure~\ref{fig:SAR_pic3}.  
We will construct a radial function $g$ supported outside the unit disc so that $R_\gamma(f-g)(r,\theta)$ is smooth in a neighborhood of $r=1/2$.

We will work with radial functions only, i.e., functions of the form $F(|x|^2)$. We will identify the latter with $F$, somewhat incorrectly. Then $R_\gamma F$ is independent of the angle $s$, and it is enough to fix $s=0$ corresponding to $x=(1,0)$. Then we have
\[
\frac1r R_\gamma F(r) = \int_{-\pi}^\pi F \left((1+r\cos\theta)^2+(r\sin\theta)^2\right)\, \d\theta.
\]
The factor $1/r$ can be explained by the requirement that the measure along each circle is Euclidean. Since $1/r$ is a smooth factor near $r=1$, we will drop it. We also use the fact that the integrand is an even function of $\theta$, so we denote $(1/(2r))R_\gamma$ by $R$: 
\[
Rf(r) = \int_{0}^\pi f \left(1+r^2+2r\cos\theta\right)\, \d\theta.
\]
Set $r=1/2+h$. We are interested in the singularities near $h=0$ and in what follows, $|h|\ll1$. After replacing $\theta$ by $\pi-\theta$, we get
\[
Rf(1/2+h)=  \int_{0}^\pi H\! \left(2r\cos\theta-r^2-3/4\right)\Big|_{r=1/2+h} \d\theta = H(h) \arccos\frac{(1/2+h)^2+3/4}{1+2h}.
\]
The following calculations were performed with Maple. The series expansion of the expression above is
\be{ex1}
Rf(1/2+h)= H(h)\left( \sqrt2h^{1/2}-\frac{17\sqrt2}{12} h^{3/2}+\frac{243\sqrt2}{160}h^{5/2}+O(h^{7/2})\right). 
\ee
We are looking for a radial $g$ of the type
\be{ex1g}
g(|x|^2) = H(t)(  a_0+a_1t+ a_2t^2+\dots  )|_{t=|x|^2-9/4}.
\ee
Then 
\[
\begin{split}
Rg(1/2+h) &= \int_{0}^\pi H(h)(a_0+a_1t+a_2t^2+\dots) 
\big|_{t =  (1+2h)\cos\theta+(1/2+h)^2-5/4 } \d\theta \\
  &= a_0A_0(h)+a_1A_1(h)+a_2A_2(h)+\cdots.
\end{split}
\]
For $A_0$ we easily get
\[
A_0(h) = \arccos\frac{5/4-(1/2+h)^2}{1+2h}= \sqrt6 h^{1/2}-\frac{7\sqrt6}{12}h^{3/2} + \frac{1243\sqrt6}{1440}h^{5/2} +O(h^{7/2}) , \quad h\ge0.
\]
By \r{ex1}, to cancel the $h^{1/2}$ term in $R(f-g)$, we need to chose
\[
a_0= \sqrt3/3.
\]
Then 
\[
R(f-g_0) = -\frac{5\sqrt2}6 h^{3/2}+O(h^{7/2}), \quad g_0 := a_0H(t)|_{t=|x|^2-9/4}.
\]
To improve the smoothness near $h=0$, we compute 
\[
A_1(h) = \frac{8\sqrt2}3 h^{3/2} -\frac{37\sqrt2}{15} h^{5/2}+O(h^{7/2}).
\]
Then, as before, we find that we need to choose
\[
a_2=-5/16
\]
to kill the  $O(h^{3/2})$ term, and then
\[
R\left(f-g_1\right) = -\frac{83\sqrt2}{720} h^{5/2}+O(h^{7/2}), \quad g_1 :=  H(t)(  a_0+a_1t)|_{t=|x|^2-9/4}.
\]
Note that this was possible to do because the leading coefficient (the one in front of $h^{3/2}$) in the expansion of $A_1$ is non-zero. The latter also follows from the ellipticity of $\Lambda$. Proceeding in the same way, we can get a full expansion of the conormal singularity of $g$ at $|x|=3/2$ that would make $R(f-g)$ smooth at $r=1/2$. 

The first three coefficients of $g$ are shown below 
\[
g= H(t)\left( \frac{\sqrt3}3 -\frac{5}{16} t+ \frac{83}{5184}t^2+ 
O\!\left(t^3\right)   \right), \quad 
t:= |x|^2-9/4.
\]

\begin{figure}[h] 
  \centering
  \includegraphics[width=3.41in,height=2.41in,keepaspectratio, trim = 40mm 0mm 0mm 00mm, clip]{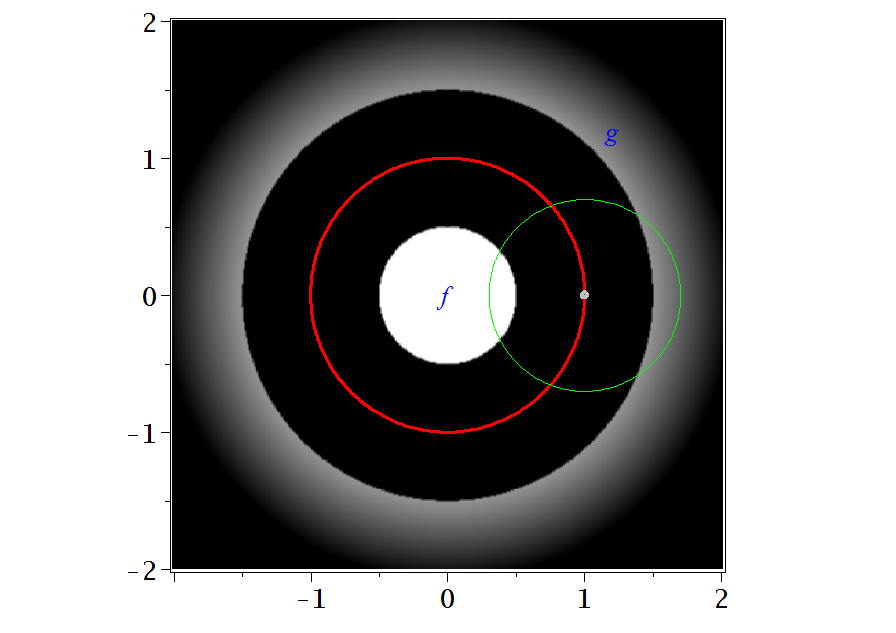}
  \includegraphics[width=3.24in,height=3.45in,keepaspectratio,trim = 20mm 110mm 30mm 70mm, clip]{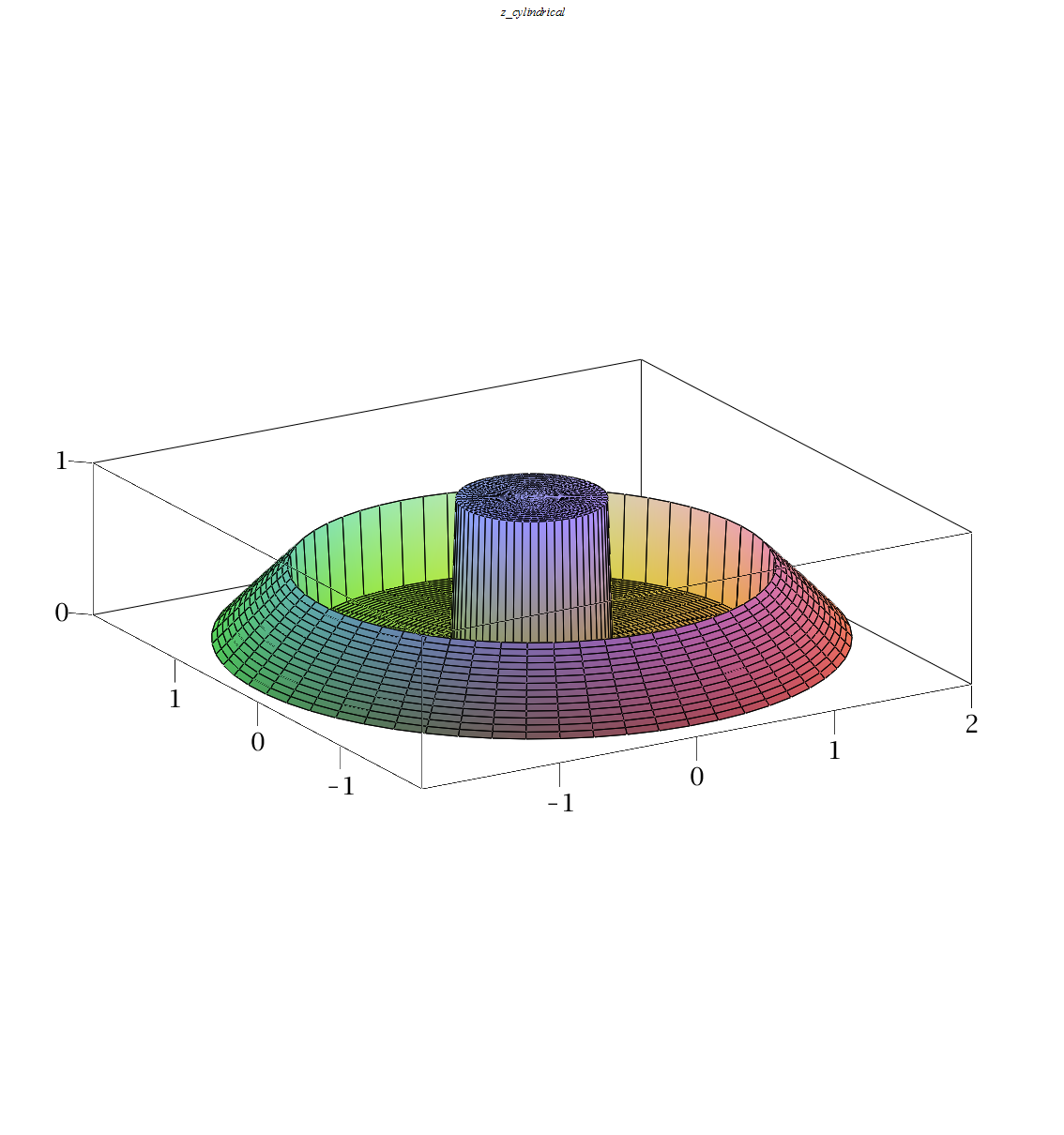} 
  \caption{Left: density plot (white $=1$, black $=0$); right: a graph of $f$ and $g$  with  $R(f-g)$ smooth near $r=1/2$.}
  \label{fig:SAR_pic1.png}
\end{figure}

\begin{figure}[h] 
  \centering
  \includegraphics[width=2.64in,height=1.58in,keepaspectratio]{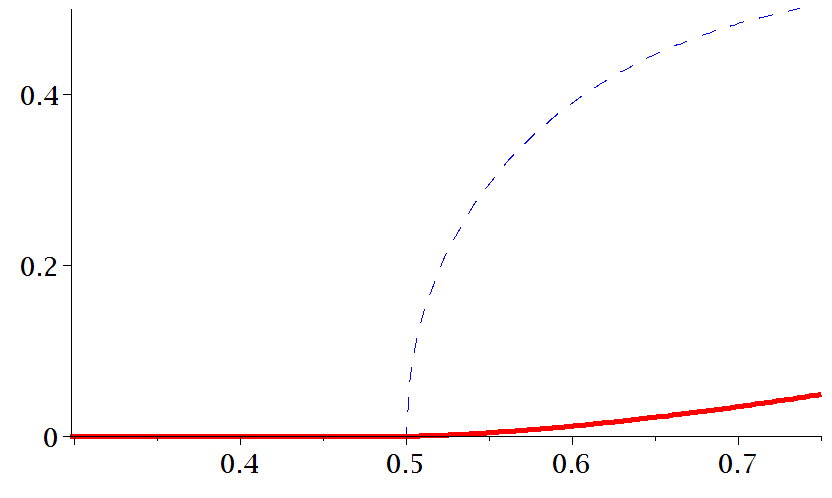}
  \caption{The thick line: The graph of $R(f-g)$ near $r=1/2$ computed numerically  with three terms in the expansion of $g$. The dotted line: the graph of $Rf$ having a square root type of singula\-rity.}
  \label{fig:SAR_pic3}
\end{figure}

We could continue this process to kill all the singularities for all $r$, not just at $r=3/2$ by constructing a suitable jump of $g$ at $r=5/2$, then at $r=7/2$, etc., which also illustrates Theorem~\ref{thm4}.

\end{example}

\subsection{Proof of Theorem~\ref{thm2}} 
Let first $(x,\xi)\in \WF(f)$. Let $\chi(x,D)$ be a pseudo-differential cutoff in some small conic neighborhood of $(x,\xi)$. Then  $R_\gamma \chi(x,D)f$ has singularities in small neighborhoods of $\mathcal{C}^\pm(x,\xi)$ in $T^*(\R_+\times \gamma)$, where we used the notation above. Take the plus sign first. Since $R_\gamma f$ is smooth, there must be another singularity that cancels this one. By section~\ref{sec_kernel}, it must be at the mirror point $(x_1,\xi_1)$ of $(x,\xi)$ about the line tangent to $\gamma$ at $p_1$, see  Figure~\ref{fig:sar2}. Clearly, $(x_1,\xi_1)$ belongs to the set $\art(x,\xi)$; it corresponds actually to the first point $t$, when $t$, increases from $t=0$, not equal to $(x,\xi)$, see \r{Art}. 
Moreover, by Theorem~\ref{thm1} we can construct $f_1$ with a wave front set near that mirror point so that $R_\gamma(\chi f+f_1)$ is smooth at  $\mathcal{C}^+(x,\xi)$. Since the line through $(x_1,\xi_1)$ crosses $\gamma$ transversely, it has to cross it again, also transversely. This creates another singularity on the boundary, represented in Figure~\ref{fig:sar2} by $p_2$ (of course, that singularity is an element of $T^*(\R_+\times \gamma)$). It needs to be canceled by another one, etc. We repeat the same argument for $\mathcal{C}^-(x,\xi)$. Therefore, we showed that if $(x,\xi)\in \WF(f)$, then the whole set $\art(x,\xi)$ is in $\WF(f)$. 

Now, assume that $(x,\xi)\not\in \WF(f)$. Since $R_\gamma f$ is smooth, a singularity at $(x_1,\xi_1)$ cannot exist, because it can only be canceled by one at $(x,\xi)$. Then we show step by step that no point in $\art(x,\xi)$ can be singular for $f$.

\subsection{Proof of Theorem~\ref{thm3}} 
Let $(x_0,\xi_0)$ be visible from $\gamma$. This means that $x_0\not\in\gamma$, and that the line through $(x_0,\xi_0)$  intersects $\gamma$ transversely. Then $(x_0,\xi_0)\in \art(x,\xi)$ for some $(x,\xi)\in T^*\Omega\setminus 0$, even if $x_0$ is outside $\Omega$ (if $x_0\in\Omega$, then $(x,\xi)= (x_0,\xi_0)$). Indeed, for this, we have to show that the equation $\Phi^{-t}\circ \Phi^t_\gamma(x,\xi)= (x_0,\xi_0)$, $x\in\Omega$, is solvable for some $t$. This is true because we can set  $(x,\xi)= \Phi^{-t}_\gamma\circ \Phi^t(x_0,\xi_0)$, where $t$ is any interior point of the non-empty interval $(t_1,t_2)$ for which $x_0+t\xi_0\in \Omega$. 
For any $t$, the projection  $\pi\circ \Phi^t_\gamma(x,\xi)$ onto the base is in $\bar\Omega$. Then $|\pi\circ \Phi^{-t}\circ \Phi^t_\gamma(x,\xi)|>|t|-C_K$, where $C_K := \max(|x|;\; x\in \bar\Omega)$. Therefore, $\art(x,\xi)$ does not lie over any compact set. By the compactness assumption of the theorem, it has elements outside $\WF(f)$. Then by Theorem~\ref{thm2}, $(x,\xi)\not\in \WF(f)$. 

\subsection{Constructing a parametrix for $R_\gamma f$ in $\Omega$, when $\WF(f)$ lies over a fixed compact set.}\label{sec_TAT} 
We will give another, constructive proof of Theorem~\ref{thm3} for the singularities of $f$ inside $T^*\Omega$. Let $\singsupp f\subset K$, where $K$ is a fixed compact set. Fix $T$ so that
\be{1.7}
T>\max\left(|x-y|;\; x\in \bo, y\in K\right). 
\ee
Then all singularities of the solution $u$ of \r{1} would leave $\bar\Omega$ for $t\ge T$, and $R_\gamma f\in C^\infty$ for $r>T$. The latter is obvious even without the propagation of singularities theory. 
 Let $v$ be the incoming solution of the wave equation in $\Omega$ with Dirichlet data $\Lambda f = (A\otimes\Id)R_\gamma f$ on $[0,T]\times\bo$, cut-off smoothly near $t=T$. More precisely, let $\chi$ be a smooth function of $t$ so that $\chi(t)=0$ for $t>T$, and $\chi(t)=1$ for $0\le t\le T_0$, where $T_0<T$ is chosen so that $T_0$ satisfies \r{1.7} as well. Let $v$ solve
\begin{equation}   \label{1.8}
\left\{
\begin{array}{rcll}
(\partial_t^2 -\Delta)v &=&0 &  \mbox{in $[0,T]\times \Omega$},\\
u|_{t=T} &=& 0  &\mbox{in  $\Omega$},\\ \quad \partial_t u|_{t=T}& =&0 &\mbox{in  $\Omega$},\\
u|_{ [0,T]\times \bo }& =& h,
\end{array}
\right.               
\end{equation}
where $h$ will be chosen in a moment to be $\chi \Lambda f $.
Set
\be{1.8'}
Gh = \partial_t v|_{t=0}.
\ee
Then $G\chi \Lambda f =f$ in $\Omega$ modulo $C^\infty$. Indeed, consider $w :=u-v$. It solves
\begin{equation}   \label{1.9}
\left\{
\begin{array}{rcll}
(\partial_t^2 -\Delta)w &=&0 &  \mbox{in $[0,T]\times \Omega$},\\
w|_{t=T} &\in& C^\infty(\Omega),\\ \quad \partial_t w|_{t=T}& \in&C^\infty(\Omega),\\
w|_{ [0,T]\times \bo }& =& (1-\chi)u,
\end{array}
\right.               
\end{equation}
Then $f-Gh=\partial_t w|_{t=0}\in C^\infty(\Omega)$, which proves our claim. 

To summarize this, we proved the following.

\begin{proposition}\label{pr_parametrix}
Let $\gamma$ be as in Theorem~\ref{thm2} and let $f\in \mathcal{D}'(\R^2)$ be such that  $\singsupp f\subset K\setminus \gamma$, where $K$ is a fixed compact set. Let $T>0$, $\chi$ be as in \r{1.7} and \r{1.8}. Then 
\[
G\chi (A\otimes\Id)R_\gamma f = f|_\Omega \quad \mod C^\infty(\Omega).
\]
\end{proposition} 
To complete the proof we only need to notice that by assumption, $\singsupp f$ is at positive distance to $\gamma=\bo$, which guarantees that $\WF(h)$, with $h= \chi (A\otimes\Id)R_\gamma f$, is separated from $t=0$, and the singularities of $w$ are never tangent to $\bo$. This makes the operator $G$ an FIO of order $0$ with a canonical relation a graph, like in the previous sections, and in particular $G$ is well defined on such $h$. Therefore, $G\chi (A\otimes\Id)R_\gamma f$ is well defined.

\subsection{Proof of Theorem~\ref{thm4}} \label{sec_TAT2}  
We first present a proof along the lines of the proof of Theorem~\ref{thm2} above. 
We  prove a somewhat weaker version first: for any $T>0$, we can complete $f$ to a distribution in $\R^2\setminus\gamma$ so that $R_\gamma f\in C^\infty((0,T)\times\gamma)$.
 Fix $(x,\xi)\in T^*\Omega\setminus 0$, and let $f$ has a wave front set in some small neighborhood of that point. Let $u_0$ be the solution of the wave equation in the plane with Cauchy data $(0,f)$. Then by section~\ref{sec_kernel}, $R_\gamma f$ will only have  singularities on $T^*\gamma$ near points defined by the line through $(x,\xi)$ which lie over $(t_{-1},p_{-1})$ and $(t_1,p_1)$, where $t_{\pm1}$ are the arrival times, see Figure~\ref{fig:sar4}. To cancel them, we chose $g_1$ with singularities near $(x_{-1},\xi_{-1})$ and $(x_1,\xi_1)$, see Figure~\ref{fig:sar4} again, unitarily related to the singularity of $f$ near $(x,\xi)$, see Theorem~\ref{thm1}. Then $u_0+u_1$, where $u_1$ is the solution with Cauchy data $(0,-g_1)$, will have no singularities near the points mentioned above which  project to $(t_{-1},p_{-1})$ and $(t_1,p_1)$. 
On the other hand, $u_1$ will  cause new singularities at points above $(t_{\pm2},p_{\pm2})$; see Figure~\ref{fig:sar4} where only $p_2$ is shown. We then construct $g_2$ and a related $u_2$ that would cancel them, etc. After a finite number of steps, the time component of the points $(t_{\pm k},p_{\pm k})$ above which  we have a singularity, will exceed $T$, and then we stop' and set $g=g_1+g_2+\dots$. Then we use a microlocal partition of unity to construct $g$ so that $f-g$ would have  the required properties without the assumption on $\WF(f)$.

To prove the general case (i.e., to take $T=\infty$ above), let $g_k$ (the subscript $k$ now has a different meaning) be the  distribution corresponding to $T=n$. Then $g_k-g_m = (f-g_m)-f-g_k$ has a circular transform smooth on $(0,\min(k,m))\times \gamma$, and $g_k-g_m =0$ in $\Omega$. The only possible singularities of that distribution could be those with the property that the line through each one of them intersects $\gamma$ transversely; then that singularity will leave a trace on $\gamma$. This implies that there are no singularities with travel time to $\gamma$ less than $\min(k,m)$. Therefore, on some ball centered at the origin of radius $\min(k,m)-C$, the distribution  $g_k$ coincides with $g_m$ up to a smooth function. Then we can easily construct $g$ as a ``limit'' of $g_k$ with a partition of unity, and this $g$ would have the property $R_\gamma(f-g)\in C^\infty$.

\begin{remark}\label{remark_TAT}  The main results in this paper are also related to the thermoacoustic/photoacoustic model with sources inside and outside $\Omega$. The wave equation then is the underlying model and there is no need of the operator $A$. Theorem~\ref{thm1} and  Theorem~\ref{thm4} then prove non-uniqueness of recovery of $\WF(f)$ as singularities of the data, with partial or full  measurements. Theorem~\ref{thm3} proves that this is actually possible if $\singsupp f$ is contained in a fixed compact set. The recovery is given by time reversal with $T$ as in \r{1.7}. The only formal difference is that in TAT, the wave equation is solved with Cauchy data $(f,0)$ at $t=0$ instead of $(0,f)$; and the time reversal operator, see \r{go2} and \r{1.8'} does not contain $\partial_t$. 

\end{remark}

\section{The 3D case: Recovery of the singularities  from integrals over spheres centered on a surface.}  \label{sec_3D} 
Let $\Gamma$ be a given smooth (relatively open) surface in $\R^3$. Let 
\be{3.1}
R_\Gamma f(r,p) = \int_{|x-p|=r} f(x) \,\d S_x,\quad r>0,  p\in\Gamma,
\ee
where $\d S_x$ is the Euclidean surface measure on the sphere $|x-p|=r$. We show below that the results of the previous section generalize easily to this case as well.

We assume again that $f\in \mathcal{E}'(\R^3)$ is supported away from $\Gamma$, and that for any  $(x,\xi)\in \WF(f)$, the line through $(x,\xi)$ hits $\Gamma$ once only, transversely. 
The main notions in section~\ref{sec2} are defined in the same way with a few minor and obvious modifications. In \r{1.W} and in the definitions \r{go1b}, \r{go1c} of $\mathcal{C}_{L,R}$ we need to replace $\dot\gamma\cdot\xi$ by the projection of $\xi$ onto the boundary,  i.e., onto $T_p^*\Gamma$, where $p\in \Gamma$ is the point where the line through $(x,\xi)$ hits $\Gamma$. 

In this case, $R_\Gamma f$  is more directly related to the solution of the wave equation; indeed
\[
u(t,x) = \frac1{4\pi t}R_\Gamma f(t,x)
\]
is the solution of the wave equation in the whole space with Cauchy data $(0,f)$ at $t=0$   restricted to $\R_+\times\Gamma$. 
Then $\Lambda = (4\pi t)^{-1}R_\Gamma$, compare with \r{1.3a}. Multiplication by $(4\pi t)^{-1}$ is, of course, an elliptic \PDO\ for $t\not=0$ (which is implied by our assumptions), and we get that Theorem~\ref{thm1} applies to this case, as well. In particular, we get that microlocally, we cannot distinguish between sources inside and outside the domain $\Omega$ occupied by the `patient's body'' in thermoacoustic tomography. If the external sources have compactly supported perturbations, then we can, and time reversal reconstruct the singularities for a large enough time $T$ such that each singularity coming from outside would exit before time $t$. 
This has been observed numerically in \cite{KuchmentKun08}.  

Finally, we remark that in applications to thermoacoustic tomography, the wave equation point of view is the natural one, actually. Then those results extend  to variable speeds using the analysis in \cite{SU-thermo}.


\end{document}